\documentclass[12pt]{article}

\usepackage{authblk}
\usepackage{cite}
\usepackage{amsmath}
\usepackage{amsthm,amsfonts,amssymb}           
\usepackage{bbm}
\usepackage[left=2.5cm,
            right=2.5cm
	       ]{geometry}
\usepackage[pdfstartview=FitH,
            CJKbookmarks=true,
            bookmarksnumbered=true,
            bookmarksopen=true,
            colorlinks,
            linkcolor=blue,
            anchorcolor=blue,
            citecolor=blue,
            urlcolor=blue
            ]{hyperref}
\numberwithin{equation}{section}

\newtheorem{thm}{Theorem~}[section]
\newtheorem{cor}[thm]{Corollary~}
\newtheorem{lem}[thm]{Lemma~}

\theoremstyle{remark}

\def\a{\alpha}
\def\bd{\operatorname{bd}}
\newcommand{\me}{e}                                     
\newcommand{\R}{\mathbb{R}}
\newcommand{\cb}{\mathcal{K}_n^n}                       
\newcommand{\lcb}{\mathcal{K}^n}                        
\newcommand{\lpoly}{\mathcal{P}^n}          		    
\newcommand{\GL}{\mathrm{GL}(n)}                        
\newcommand{\intr}{\mathrm{int}\,}

\newcommand{\rmnum}[1]{\romannumeral #1}
\newcommand{\Rmnum}[1]{\expandafter\@slowromancap\romannumeral #1@}
\newcommand{\norm}[1]{\left\Vert{#1}\right\Vert}        
\newcommand{\abs}[1]{\left\vert{#1}\right\vert}         
\newcommand{\xabs}[1]{\vert{#1}\vert}                   
\newcommand{\set}[1]{\left\{{#1}\right\}}               
\newcommand{\xset}[1]{\{#1\}}                           
\newcommand{\inp}[2]{#1\cdot#2}                         
\newcommand{\conr}{C(\R^n)}                             
\newcommand{\la}{L^1_c(\R^n)}                           
\newcommand{\kik}[1]{\mathbbm{1}_{#1}}                  
\newcommand{\upar}{\mathrm{Par}(n)}                     
\newcommand{\lt}{\mathcal{L}}                           
\newcommand{\fijk}{f_{i_{j_k}}}

\makeatletter
\newcommand{\subjclass}[2][1991]{%
  \let\@oldtitle\@title%
  \gdef\@title{\@oldtitle\footnotetext{#1 \emph{Mathematics subject classification.} #2}}%
}
\newcommand{\keywords}[1]{%
  \let\@@oldtitle\@title%
  \gdef\@title{\@@oldtitle\footnotetext{\emph{Key words and phrases.} #1.}}%
}
\makeatother

\title{\bf{Laplace transforms and valuations}}
\author[1]{Jin Li}
\author[2]{Dan Ma}
\affil[1]{Department of Mathematics, Shanghai University, Shanghai 200444, China\\\href{mailto: Jin Li<lijin2955@gmail.com>}{lijin2955@gmail.com}}
\affil[2]{Institut f\"{u}r Diskrete Mathematik und Geometrie, Technische Universit\"{a}t Wien, Wiedner Hauptstra\ss e 8--10/104, 1040 Wien, Austria\\\href{mailto: Dan Ma<madan516@gmail.com>}{madan516@gmail.com}}
\date{}
\subjclass[2010]{44A10, 52A20, 52B45}
\keywords{Laplace transform, valuation, $\GL$ covariance, logarithmic translation covariance}

\begin{document}

\maketitle

\begin{abstract}
It is proved that the classical Laplace transform is a continuous valuation which is positively $\GL$ covariant
and logarithmic translation covariant. Conversely, these properties turn out to be sufficient to characterize this transform.
\end{abstract}


\section{Introduction}

Let $f:[0,\infty)\to\R$ be a measurable function. The \emph{Laplace transform} of $f$ is given by
\[\lt f(s)=\int_0^\infty e^{-st}f(t)dt, ~~s \in \R\]
whenever the integral converges.
In the 18th century, Euler first considered this transform
to solve second-order linear ordinary differential equations
with constant coefficients. One hundred years later,
Petzval and Spitzer named this transform after Laplace.
Doetsch \text{initiated} systematic investigations in 1920s.
The Laplace transform now is widely used for solving
ordinary and partial differential equations.
Therefore, it is a useful tool not only for \text{mathematicians} but also
for physicists and engineers (see, for example, \cite{Doe74}).

The Laplace transform has been generalized to the multidimensional setting
in order to solve ordinary and partial differential equations in boundary value problems
of several variables (see, for example, \cite{Deb88}).
Let $f$ be a compactly supported function that belongs to $L^1(\R^n)$.
The multidimensional Laplace transform of $f$ is defined as
\[\lt f(x)=\int_{\R^n}f(y)e^{-\inp{x}{y}}dy, ~~x\in\R^n.\]

The Laplace transform is also considered
on $\cb$, the set of $n$ dimensional convex bodies (i.e., compact convex sets) in $\R^n$.
The \emph{Laplace transform} of $K\in \cb$ is defined by
\[\lt K(x) = \lt (\kik{K})(x)=\int_Ke^{-\inp{x}{y}}dy,~~x\in \R^n,\]
where $\kik{K}$ is the indicator function of $K$.
Making use of the logarithmic version of this transform,
Klartag \cite{Kla06} improved Bourgain's esimate
on the slicing problem (or hyperplane conjecture),
which is one of the main open problems in the asymptotic theory
of convex bodies. It asks whether every convex body
of volume 1 has a hyperplane section through the origin
whose volume is greater than a universal constant
(see also \cite{KM12} for more information).

Noticing that both Laplace transforms are valuations,
we aim at a deeper understanding on these classical integral transforms.
A function $z$ defined on a lattice $(\Gamma,\vee,\wedge)$
and taking values in an abelian semigroup is called a \emph{valuation} if
\begin{equation}\label{eqn:dval}
    z(f\vee g)+z(f\wedge g)=z(f)+z(g)
\end{equation}
for all $f,g\in\Gamma$. A function $z$ defined on some subset $\Gamma_0$ of $\Gamma$
is called a valuation on $\Gamma_0$ if \eqref{eqn:dval} holds whenever
$f,g,f\vee g,f\wedge g\in\Gamma_0$.
Valuations were a key ingredient in Dehn's solution of Hilbert's Third Problem in 1901.
They are closely related to dissections and lie at the very heart of geo\-metry.
Here, valuations were considered on the space of convex bodies in $\R^n$, denoted by $\lcb$.
Perhaps the most famous result is Hadwiger's characterization theorem
which classifies all continuous and rigid motion invariant real valued valuations on $\lcb$.
Klain \cite{Kla95} provided a shorter proof of this beautiful result
based on the following characterization of the volume.

\begin{thm}[\cite{Had57,Kla95}]\label{thm:Klain}
    Suppose $\mu$ is a continuous rigid motion invariant and simple valuation on $\lcb$.
    Then there exists $c\in\R$ such that $\mu(K)=cV_n(K)$, for all $K\in\lcb$.
    Here, $V_n$ is the $n$ dimensional volume.
\end{thm}

Other important later contributions can be found in \cite{Had57,KR97,McM93,McS83}.
For more recent results, we refer to
\cite{Ale99,Ale01,Hab09,Hab12a,Hab12b,HP14a,HP14b,HL06,Kla96,Kla97,LYL15,Lud02b,Lud03,Lud06,LR99,LR10,Par14a,Par14b,SS06,Sch08,SW12,Wan11}.

With the first result of this paper, we characterize the Laplace transform on convex bodies.

\begin{thm}\label{thm:con}
A map $Z: \cb \to C(\R^n)$ is a continuous, positively $\GL$ covariant and logarithmic translation covariant valuation
if and only if there exists a constant $c \in \R$ such that
\begin{align*}
ZK=c \mathcal {L}K
\end{align*}
for every $K \in \cb$.
\end{thm}

Throughout this paper, without further remark, we briefly write positively $\GL$ covariant as $\GL$ covariant;
see Section \ref{sec:pre} for definitions of $\GL$ covariance and logarithmic translation covariance.
We call $Z: \cb \to C(\R^n)$ \emph{continuous} if for every $x\in \R^n$, we have $ZK_i(x) \to ZK(x)$ whenever $K_i \to K$ with respected to the Hausdorff metric,
where $K_i,K \in \cb$.

Notice that $\lt K(0)=V_n(K)$ holds for all $K\in\cb$.
Thus, this characterization is a generalization of Theorem \ref{thm:Klain}.

Valuations are also considered on spaces of real valued functions.
Here, we take the pointwise maximum and minimum as the join and meet, respectively.
Since the indicator functions of convex bodies provide a one-to-one correspondence with convex bodies,
valuations on function spaces are generalizations of valuations on convex bodies.
Valuations on function spaces have been studied since 2010.
Tsang \cite{Tsa10} characterized real valued valuations on $L^p$-spaces.
Kone \cite{Kon14} generalized this characterization to Orlicz spaces.
As for valuations on Sobolev spaces, Ludwig \cite{Lud11b,Lud12} characterized
the Fisher information matrix and the Lutwak-Yang-Zhang body.
Other recent and interesting characterizations can be found in
\cite{Cav15,CC15,Tsa12,BGW13,Lud13,Ma15,Wang14,Obe14}.

With the second result of this paper, we characterize the Laplace transform on functions
based on Theorem \ref{thm:con} and the natural connection between
indicator functions and convex bodies.
Let $\la$ denote the space of compactly supported functions that belong to $L^1(\R^n)$.
We call $z:\la\rightarrow\conr$ \emph{continuous} if for every $x\in \R^n$,
we have $z(f_i)(x) \to z(f)(x)$ whenever $f_i\to f$ in $L^1(\R^n)$.

\begin{thm}\label{thm:laplace}
    A map $z:\la\rightarrow\conr$ is a continuous,
    positively $\GL$ covariant and logarithmic translation covariant valuation if and only if
    there exists a continuous function $h$ on $\R$ with the properties that
    \begin{equation}\label{eqn:zero}
    	h(0)=0
    \end{equation}
    and that there exists a constant $\gamma\geq0$ that
    \begin{equation}\label{eqn:growth}
    	\abs{h(\a)}\leq\gamma\abs{\a}
    \end{equation}
    for all $\a\in\R$,
    such that
    \[z(f) = \lt(h\circ f)\]
    for every $f\in\la$.
\end{thm}

If we further assume that $z$ is $1$-homogeneous, that is, $z(sf)=s z(f)$
for all $s\in\R$ and $f\in\la$, then we obtain the Laplace transform.

\begin{cor}
    A map $z:\la\rightarrow\conr$ is a continuous, $1$-homogeneous, positively
    $\GL$ covariant and logarithmic translation covariant valuation if and only if
    there exists a constant $c\in\R$ such that
    $$z(f)=c\lt f,$$ for every $f\in\la$.
\end{cor}


\section{Preliminaries and Notation}
\label{sec:pre}

Our setting is the $n$-dimensional Euclidean space $\R^n$ with the standard basis $\{e_1,\dots,e_n\}$, where $n \geq 1$.
The convex hull of a set $A$ is denoted by $[A]$ and the convex hull of a set $A$ and a point $x \in \R^n$ will be briefly written as $[A,x]$ instead of $[A,\{x\}]$.
A hyperplane is an $n-1$ dimensional affine space in $\R^n$.
The unit cube $C^n=\sum_{1\leq i \leq n}[o,e_i]$ and the standard simplex $T^n=[o,e_1,\dots,e_n]$ are two important convex bodies in this paper.

The Hausdorff distance of $K,L \in \lcb$ is
\begin{align*}
d(K,L) = \inf \{\varepsilon >0:  K \subset L + \varepsilon B, ~~L \subset K + \varepsilon B \}.
\end{align*}

The norm on the space $\la$ is the ordinary $L^1$ norm which is denoted by $\norm{\cdot}$.

A map $z:\la \to \conr$ is called \emph{$\GL$ covariant} if
\begin{equation}\label{eqn:gln}
    z(f\circ\phi^{-1})(x)=\abs{\det\phi}z(f)(\phi^t x)
\end{equation}
for all $f \in \la$, $\phi\in\GL$ and $x\in\R^n$. In this paper, we actually deal with positive $\GL$ covariance,
that is \eqref{eqn:gln} is supposed to hold for all $\phi \in \GL$ that have positive determinant.
Also, a map $z:\la \to \conr$ is called \emph{logarithmic translation covariant} if
\begin{align*}
z(f(\cdot-t))(x) =& \me^{-\inp{x}{t}} z (f)(x)
\end{align*}
for all $f\in\la$ and $t,x\in\R^n$. This definition is motivated by the relation
\begin{align*}
\log \lt (f(\cdot-t))(x) = -\inp{x}{t} +  \log \lt f(x)
\end{align*}
(see Theorem \ref{thm:prop}).

A map $Z:\cb \to \conr$ is called \emph{$\GL$ covariant} if
\[Z(\phi K)(x)=\abs{\det\phi}ZK(\phi^t x)\]
for all $K \in \cb$, $\phi\in\GL$ and $x\in\R^n$.
Also, a map $z:\cb \to \conr$ is called \emph{logarithmic translation covariant} if
\begin{align*}
Z(K+t)(x)=\me^{-t \cdot x}ZK(x)
\end{align*}
for all $K\in\cb$ and $t,x\in\R^n$. Again, it is motivated by the relation
\begin{align*}
\log \lt (K+t)(x) = -t \cdot x + \log \lt K(x)
\end{align*}
(see Theorem \ref{thm:lapcon}).
If a valuation vanishes on lower dimensional convex bodies, we call it \emph{simple}.

As we will see in Lemma \ref{lem:rel}, if $z:\la \to \conr$ is continuous, $\GL$ covariant and logarithmic translation covariant, then $Z:\cb \to \conr$ defined by $ZK = z(\kik{K})$ is also continuous, $\GL$ covariant and logarithmic translation covariant, respectively.

For the constant zero function, if $z:\la \to \conr$ is $\GL$ covariant and logarithmic translation covariant, then
\begin{align}\label{eq:z0}
z(0) \equiv 0.
\end{align}
Indeed, $z(0)(\phi^t x)= z(0)(x)$ for any $\phi \in \GL$. Let $x=e_1$. We have that $z(0)$ is a constant function on $\R^n \setminus \{0\}$. The continuity of the function $z(0)$ now gives that $z(0) \equiv c$ on $\R^n$ for a constant $c \in \R$. Since $z$ is also logarithmic translation covariant, $z(0)(x) = e^{-t \cdot x} z(0)(x)$ for any $x,t \in \R^n$. Hence $z(0) \equiv 0$.


\section{Laplace transforms}

In this section, we study some properties of Laplace transforms.

\begin{thm}\label{thm:prop}
    Let $h$ be a continuous function on $\R$ satisfying \eqref{eqn:zero} and \eqref{eqn:growth}.
    If a map $z:\la\rightarrow\conr$ satisfies
    \[z(f)(x)=\int_{\R^n}(h\circ f)(y)\me^{-\inp{x}{y}}dy\]
    for every $x\in\R^n$ and $f\in\la$,
    then $z$ is a continuous, $\GL$ covariant and logarithmic translation covariant valuation.

    In particular, if $h(\a)=\a$ for all $\a\in\R$, the Laplace transform $\mathcal {L}$ on $\la$ is a continuous, $\GL$ covariant and logarithmic translation covariant valuation.
\end{thm}
\begin{proof}
    Let $f,g\in\la$ and $E=\set{x\in\R^n:f(x)\leq g(x)}$. Then
    \begin{align*}
    	z(f\vee g)(x) =& \int_{\R^n}h\circ(f\vee g)(y)e^{-\inp{x}{y}}dy\\
	   =& \int_E(h\circ g)(y)e^{-\inp{x}{y}}dy+\int_{\R^n\setminus E}(h\circ f)(y)e^{-\inp{x}{y}}dy
    \end{align*}
    for every $x\in\R^n$. Similarly, we have
    \begin{align*}
    	z(f\wedge g)(x) =& \int_{\R^n}h\circ(f\wedge g)(y)e^{-\inp{x}{y}}dy\\
	   =& \int_E(h\circ f)(y)e^{-\inp{x}{y}}dy+\int_{\R^n\setminus E}(h\circ g)(y)e^{-\inp{x}{y}}dy
    \end{align*}
    for every $x\in\R^n$.  Thus,
    \begin{align*}
    	& z(f\vee g)(x) + z(f\wedge g)(x)\\
	=& \int_{\R^n}(h\circ f)(y)e^{-\inp{x}{y}}dy + \int_{\R^n}(h\circ g)(y)e^{-\inp{x}{y}}dy\\
	=& z(f)(x) + z(g)(x)
    \end{align*}
    for every $x\in\R^n$.

    Next, we are going to show that $z$ is continuous.
    Let $f\in\la$ and let $\set{f_i}$ be a sequence in $\la$ that converges to $f$ in $L^1(\R^n)$.
    We will show the continuity of $z$ by showing that
    for every subsequence $\xset{z(f_{i_j})(x)}$ of $\set{z(f_i)(x)}$,
    there exists a subsequence $\xset{z(\fijk)(x)}$ that converges to $z(f)(x)$
    for every $x\in\R^n$.

Let $\xset{f_{i_j}}$ be a subsequence of $\set{f_i}$ and $y\in\R^n$.
Then, for every $x \in \R^n$, the sequence of functions $y \mapsto f_{i_j}(y) e^{-x \cdot y}$ converges to the function $y \mapsto f(y) e^{-x \cdot y}$ as $j \to \infty$ with respect to the $L^1$ norm.
It follows that there exists a subsequence $\xset{\fijk}$ of $\xset{f_{i_j}}$ and a nonnegative function $F_x \in L^1(\R^n)$ such that\\
(i) $\fijk(y)e^{-x \cdot y} \to f(y)e^{-x \cdot y}$ almost every $y$ with respect to Lebesgue measure; \\
(ii) $\xabs{\fijk(y)} e^{-x \cdot y} \leq F_x(y)$ almost every $y$ with respect to Lebesgue measure (see \cite[Section 2.7]{LL01}).
Since $h$ is continuous, we obtain
\[h\circ\fijk \to h\circ f \quad\text{a.e.}\]
Also since $h$ satisfies \eqref{eqn:growth}, we have
    \begin{align*}
    |h \circ \fijk (y)| \leq \gamma |\fijk(y)| \leq \gamma F_x(y) e^{x \cdot y}.
    \end{align*}
Note that $\int_{\R^n} F_x (y)e^{x \cdot y} \cdot e^{-x \cdot y}dy < \infty$.
We conclude from the dominated convergence theorem that
\begin{align*}
z(f)(x) = \int_{\R^n}(h\circ f)(y)\me^{-\inp{x}{y}}dy = \lim\limits_{k \to \infty}  \int_{\R^n} (h\circ \fijk)(y)\me^{-\inp{x}{y}}dy = \lim\limits_{k \to \infty} z(\fijk)(x).
\end{align*}

    Moreover, for $\phi\in\GL$,
    \begin{align*}
        z(f\circ\phi^{-1})(x) =& \int_{\R^n}(h\circ f\circ\phi^{-1})(y)\me^{-\inp{x}{y}}dy\\
        =& \abs{\det\phi}\int_{\R^n}(h\circ f)(w)\me^{-\inp{x}{(\phi w)}}dw\\
        =& \abs{\det\phi}\int_{\R^n}(h\circ f)(w)\me^{-\inp{(\phi^tx)}{w}}dw\\
        =& \abs{\det\phi}z(f)(\phi^tx)
    \end{align*}
    for every $x\in\R^n$ and $f\in\la$. Finally, let $t\in\R^n$. Then
    \begin{align*}
        z(f(\cdot-t))(x) =& \int_{\R^n}(h\circ f)(y-t))\me^{-\inp{x}{y}}dy\\
        =& \int_{\R^n}(h\circ f)(w)\me^{-\inp{x}{(w+t)}}dw\\
        =& \me^{-\inp{x}{t}}\int_{\R^n}(h\circ f)(w)\me^{-\inp{x}{w}}dw\\
        =& \me^{-\inp{x}{t}}z(f)(x)
    \end{align*}
    for every $x\in\R^n$ and $f\in\la$.
\end{proof}

Next, we turn to the Laplace transform on convex bodies.
\begin{lem}\label{lem:rel}
If $z: \la \to \conr$ is a continuous, $\GL$ covariant and logarithmic translation covariant valuation, then for any $\alpha \in \R$, $Z: \cb \to \conr$ defined by
$$ZK = z (\alpha \kik{K})$$
is a continuous, $\GL$ covariant and logarithmic translation covariant valuation on $\cb$.
\end{lem}
\begin{proof}
    For $K,L,K\cup L, K \cap L\in\cb$, we have
    \begin{align*}
        Z(K\cup L)+Z(K\cap L) &= z(\alpha \kik{K\cup L})+z(\alpha \kik{K\cap L})\\
        &= z((\alpha \kik{K})\vee (\alpha \kik{L}))+z((\alpha \kik{K})\wedge (\alpha \kik{L}))\\
        &= z(\alpha \kik{K})+z(\alpha \kik{L})\\
        &= ZK+ZL.
    \end{align*}
    Also, for a sequence $\set{K_i}$ in $\cb$ that converges to $K\in\cb$
    as $i\rightarrow\infty$, we have $\norm{\alpha \kik{K_i} - \alpha \kik{K}} \to 0$ as $i\rightarrow\infty$.
    Indeed, for any $0 < \varepsilon \leq 1$, we have
    \begin{align*}
    K_i \subset K + \varepsilon B, ~~K \subset K_i + \varepsilon B
    \end{align*}
    for sufficiently large $i$. Hence $(K_i \setminus K) \cup (K \setminus K_i) \subset \{x :~ \exists~ y \in \bd K,~\text{s.t.}~ d(x,y)\leq \varepsilon\}$, where $\bd K$ is the boundary of $K$. Hence,
    \begin{align*}
    \int_{\R^n} |\alpha \kik{K_i}(y)- \alpha \kik{K}(y)|dy
    &\leq |\alpha| V_n(\{x :~ \exists~ y \in \bd K,~\text{s.t.}~ d(x,y)\leq \varepsilon\}) \\
    &\leq |\alpha|  \cdot 2 \varepsilon S(K+B)
    \end{align*}
    for sufficiently large $i$. Here, $S$ denotes the surface area.
    By the continuity of $z$ on $\la$, we obtain
    \[Z(K_i)=z(\alpha \kik{K_i})\rightarrow z(\alpha \kik{K})=ZK\]
    as $i\rightarrow\infty$. Moreover, for each $\phi\in\GL$ and $K\in\cb$, we have
    \begin{align*}
        Z(\phi K) &= z(\alpha \kik{\phi K})=z(\alpha \kik{K}\circ\phi^{-1})\\
        &= \abs{\det\phi}z(\alpha \kik{K})\circ\phi^t=\abs{\det\phi}ZK\circ\phi^t.
    \end{align*}
    Finally, for each $t,x\in\R^n$, we have
    \begin{align*}
        Z(K+t)(x) &= z(\alpha \kik{K+t})(x)=z(\alpha \kik{K}(\cdot- t))(x)\\
        &= \me^{-\inp{t}{x}}z(\alpha \kik{K})(x)=\me^{-\inp{t}{x}}ZK(x).
    \end{align*}
\end{proof}

The following theorem directly follows from the definition of the Laplace transform, Theorem \ref{thm:prop} and Lemma \ref{lem:rel}.
\begin{thm}\label{thm:lapcon}
The Laplace transform on $\cb$ is a continuous, $\GL$ covariant and logarithmic translation covariant valuation.
\end{thm}


\section{Characterizations of Laplace transforms}
In this section, we first characterize the Laplace transform on $\cb$ as a
continuous, $\GL$ covariant and logarithmic translation covariant valuation.
Afterwards, via an approach developed from Tsang's in \cite{Tsa10},
we further characterize the Laplace transform on $\la$.

\subsection{The Laplace transform on convex bodies}\label{sec:con}
We first need to extend the valuation to $\lcb$.
\begin{lem}\label{lem:ext}
If $Z: \cb \to C(\R^n)$ is a continuous, $\GL$ covariant and logarithmic translation covariant valuation, then $\overline{Z} : \lcb \to C(\R^n)$ defined by
\begin{align*}
\overline{Z} K(x) =\begin{cases}
ZK(x), & \dim K = n, \\
0, & \dim K <n
\end{cases}
\end{align*}
is a simple, $\GL$ covariant and logarithmic translation covariant valuation on $\lcb$.
\end{lem}
\begin{proof}
The $\GL$ covariance and the logarithmic translation covariance are trivial. It remains to show that
\begin{align}\label{valext}
ZK(x) = Z(K \cap H^+)(x) + Z(K \cap H^-)(x), ~~x \in \R^n
\end{align}
for every hyperplane $H$ (when $n=1$, $H$ is a single point) such that $K, K \cap H^+, K \cap H^- \in \cb$. Since $Z$ is logarithmic translation covariant, we can assume w.l.o.g. that $o \in (\intr K \cap H)$. We can further assume that $e_n \perp H$ and $e_n \in H^+$ due to the $\GL$ covariace of $Z$.
For a fixed $K$, note that $\pm se_n \in K$ for sufficiently small $s>0$. Hence the valuation property of $Z$ shows that
\begin{align}\label{val4}
ZK(x) + Z[K \cap H, \pm se_n](x) = Z[K \cap H^+, -se_n](x) + Z[K \cap H^-, se_n](x)
\end{align}
for every $x \in \R^n$ and sufficiently small $s>0$. The $\GL$ covariance of $Z$ gives that
\begin{align*}
Z[K \cap H, \pm se_n](x) = s Z[K \cap H, \pm e_n](x_1e_1 + \dots + x_{n-1}e_{n-1} + sx_ne_n),
\end{align*}
where $x=x_1e_1 + \dots + x_{n}e_{n} \in \R^n$.
Since
\begin{align*}
&\lim\limits_{s \to 0^+} Z[K \cap H, \pm e_n](x_1e_1 + \dots + x_{n-1}e_{n-1} + sx_ne_n) \\
&= Z[K \cap H, \pm e_n](x_1e_1 + \dots + x_{n-1}e_{n-1}),
\end{align*}
we have
\begin{align}\label{lim}
\lim\limits_{s \to 0^+} Z([K \cap H, \pm se_n])(x) \to 0.
\end{align}
Now combing (\ref{val4}) and (\ref{lim}) with the continuity of $Z$, we get that (\ref{valext}) holds true.
\end{proof}

Next we consider $ZC^n$.
\begin{lem}\label{lem:cube}
If $Z: \lcb \to C(\R^n)$ is a simple, $\GL$ covariant and logarithmic translation covariant valuation, then there exists a constant $c \in \R$ such that
\begin{align}
ZC^n(re_1)=c \mathcal{L}C^n(re_1) = c \int_{C^n} \me^{-re_1 \cdot y} dy, \label{cube}
\end{align}
for every $r \in \R$.
\end{lem}
\begin{proof}
First note that
\begin{align}\label{cube3}
\int_{C^n} \me^{-re_1 \cdot y} dy = \frac{1}{r}(1-\me^{-r}).
\end{align}
For $s>0$, let $\psi_s \in\GL$ such that $\psi_s e_1=se_1$ and $\psi_s e_k=e_k$ for $2\leq k\leq n$. For integers $p,q>0$, since $Z$ is simple, we have
\[Z(\psi_{q/p}C^n) (e_1)=\sum_{j=0}^{q-1} Z \left(\psi_{1/p}C^n+\frac{je_1}{p}\right) (e_1).\]
Also since $Z$ is $\GL$ and logarithmic translation covariant, we have
\begin{align*}
\frac{q}{p} ZC^n\left(\frac{q}{p}e_1\right) &= \frac{1}{p}\sum_{j=0}^{q-1}\me^{-j/p}ZC^n\left(\frac{e_1}{p}\right)\\
&= \frac{1}{p}ZC^n\left(\frac{e_1}{p}\right)\frac{1-\me^{-q/p}}{1-\me^{-1/p}}.
\end{align*}
In particular, if $q=p$, we have
\[Z C^n\left(\frac{e_1}{p}\right)=\frac{p(1-\me^{-1/p})}{1-\me^{-1}}Z C^n(e_1).\]
Combining the two formulas above with (\ref{cube3}), and letting $c=\frac{ZC^n(e_1)}{1-\me^{-1}}$, (\ref{cube}) holds for $r = q/p$. Now since $Z C^n$ is a continuous function on $\R^n$, (\ref{cube}) holds for $r \geq 0$.

For $r<0$. Repeating the same process for $-e_1$, we obtain
\begin{align*}
\frac{q}{p} ZC^n\left(-\frac{q}{p}e_1\right) &= \frac{1}{p}\sum_{j=0}^{q-1}\me^{j/p}ZC^n\left(-\frac{e_1}{p}\right)\\
&= \frac{1}{p}ZC^n\left(-\frac{e_1}{p}\right)\frac{1-\me^{q/p}}{1-\me^{1/p}}
\end{align*}
and
\[Z C^n\left(-\frac{e_1}{p}\right)=\frac{p(1-\me^{1/p})}{1-e}Z C^n(-e_1).\]
Combining the two formulas above and letting $c'=-\frac{ZC^n(-e_1)}{1-e}$, we have
$$ZC^n(re_1) = c'\int_{C^n} \me^{-re_1 \cdot y} dy$$
for $r=-q/p$. The continuity of the function $ZC^n$ gives that $c=c'$ and thus (\ref{cube}) holds for $r \leq 0$.
\end{proof}

Now we consider valuations on polytopes.
Let $\lpoly$ denote the set of polytopes in $\R^n$ and let $Z: \lpoly \to C(\R^n)$ be a valuation.
The inclusion-exclusion principle states that $Z$ extends uniquely to $U(\mathcal {P})$, the set of finite unions of polytopes, with
$$Z(P_1 \cup \dots \cup P_m) = \sum_{1\leq j \leq m} (-1)^{j-1} \sum_{1 \leq i_1 < \dots <i_j \leq m} Z(P_{i_1} \cap \dots \cap P_{i_j})$$
for every $P_1, \dots, P_m \in \lpoly$ (see \cite{LR06} or \cite[Theorem 6.2.1 and Theorem 6.2.3]{Sch14}).

\begin{lem}\label{lem:inex}
If $Z: \lpoly \to C(\R^n)$ is a simple, $\GL$ covariant and logarithmic translation covariant valuation such that
\begin{align}\label{cube2}
ZC^n(re_1)=0
\end{align}
for every $r \in \R$, then
$$ZT^n(re_1)=0$$
for every $r \in \R$.
\end{lem}
\begin{proof}
The case $n = 1$ is trivial. We only consider $n \geq 2$.

First we prove that $ZT^n(o)=0$. Since $C^n = \bigcup_{1\leq i_1<\dots <i_n \leq n} \{x \in \R^n: 0 \leq x_{i_1} \leq \dots \leq x_{i_n} \leq 1 \}$, and all the sets $\{x \in \R^n: 0 \leq x_{i_1} \leq \dots \leq x_{i_n} \leq 1 \}$ are $\GL$ transform  (with positive determinant) images of $T^n$, the valuation property, simplicity, and $\GL$ covariance of $Z$ combined with \eqref{cube2} give that $$ZT^n(o) =0.$$

Next we deal with the case $r \neq 0$. Let $g(m,s) = Z(mT^n)(se_1)$ for $s \in \R$ and integer $m \geq 0$.
For integer $k \geq 1$, denote $M_k^n = kT^n \cap C^n$. Note that when $k \geq n$, we have
\begin{align}\label{k>=n}
M_k^n = C^n.
\end{align}
For $1 \leq k \leq n-1$,
\begin{align}\label{union}
kT^n \cup C^n = \left(\bigcup_{j=1}^{n}\left( kT^n \cap \{x \in \R^n: x_j \geq 1\}\right)\right) \cup C^n
\end{align}
Denote $T_j = kT^n \cap \{x \in \R^n: x_j \geq 1\}$. We have
\begin{align*}
T_j = (k-1)T^n + e_j,
\end{align*}
and
\begin{align*}
T_{j_1} \cap \dots \cap T_{j_i} = (k-i)T^n + e_{j_1} + \dots + e_{j_i}
\end{align*}
for $i \leq k$ and $1 \leq j_1 < \dots <j_i \leq n$.
Hence, the valuation property (after extension), inclusion-exclusion principle, simplicity and logarithmic translation covariance of $Z$, combined with (\ref{cube2}) and (\ref{union}), give that
\begin{align}\label{k<n}
Z(M_k^n)(se_1)
&= Z(kT^n)(se_1) - Z(kT^n \cup C^n)(se_1) \nonumber\\
&= Z(kT^n)(se_1) - \left(\sum_{i=1}^{k-1} (-1)^{i-1} \sum_{1 \leq j_1 < \dots < j_i \leq n} Z(T_{j_1} \cap \dots \cap T_{j_i})(se_1) \right)                  \nonumber\\
&= Z(kT^n)(se_1) - \left( \sum_{i=1}^{k-1} (-1)^{i-1} \left( \binom{n-1}{i-1} e^{-s} + \binom{n-1}{i} \right) Z((k-i)T^n)(se_1) \right) \nonumber \\
&= \sum_{i=0}^{k-1} (-1)^{i} a_i(s) g(k-i,s),
\end{align}
where $1 \leq k \leq n-1$, $a_i(s) = \binom{n-1}{i-1} e^{-s} + \binom{n-1}{i}$ for $1 \leq i \leq n-1$, and $a_0(s) = 1$.

For non negative integers $k_1, \dots , k_n$ satisfying $k = k_1 + \dots + k_n \leq m-1$, we have
$$mT^n \cap (C^n + k_1e_1 + \dots + k_ne_n)  =  M_{m-k}^n + k_1e_1 + \dots + k_ne_n.$$
For $m \geq n$, applying the valuation property, simplicity, logarithmic translation covariance of $Z$, (\ref{cube2}), (\ref{k>=n}) and (\ref{k<n}), we have
\begin{align}\label{eqn:gef}
g(m,s) &= Z(mT^n)(se_1) \nonumber \\
&=\sum_{k=0}^{m-1} \sum_{k_1 + \dots +k_n= k, \atop  k_1, \dots, k_n \geq 0} Z(mT^n \cap (C^n + k_1e_1 + \dots + k_ne_n))(se_1) \nonumber \\
&= \sum_{k=0}^{m-1} \sum_{k_1 + \dots +k_n= k, \atop  k_1, \dots, k_n \geq 0} Z(M_{m-k}^n + k_1e_1 + \dots + k_ne_n)(se_1) \nonumber \\
&= \sum_{k=0}^{m-1} \sum_{k_1 =0}^k e^{-k_1 s} \sum_{k_2 + \dots + k_n= k-k_1, \atop  k_2, \dots, k_n \geq 0} Z(M_{m-k}^n)(se_1) \nonumber\\
&= \sum_{k=0}^{m-1} Z(M_{m-k}^n)(se_1) \sum_{k_1 =0}^k e^{-k_1 s} \binom{k-k_1 + n-2}{n-2} \nonumber\\
&= \sum_{k=m-n+1}^{m-1} \left(\sum_{i=0}^{m-k-1} (-1)^i a_i(s) g(m-k-i,s) \right) \left( \sum_{k_1 =0}^k e^{-k_1 s} \binom{k-k_1 + n-2}{n-2}\right) \nonumber \\
&= \sum_{j=1}^{n-1} b_j (m,s) g(j,s),
\end{align}
where $b_j (m,s) = \sum\limits_{k=m-n+1}^{m-j} (-1)^{m-k-j} a_{m-k-j}(s) \sum\limits_{k_1 =0}^k e^{-k_1 s} \binom{k-k_1 + n-2}{n-2}$ for $1 \leq j \leq n-1$.

Since $Z$ is $\GL$ covariant, we have
\begin{align*}
g(m,s) = m^n g(1,ms).
\end{align*}
Hence the equation (\ref{eqn:gef}) gives that
\begin{align}\label{eqn:gef2}
g(1,ms) = \sum_{j=1}^{n-1} (j/m)^n b_j (m,s) g(1,js).
\end{align}

For any fixed $r \neq 0$, taking $s=r/m$ in (\ref{eqn:gef2}), we get
\begin{align}
g(1,r) = \sum_{j=1}^{n-1} (j/m)^n b_j (m,r/m) g(1,jr/m).
\end{align}
Since $g(1,\cdot)$ is a continuous function and $g(1,0) = ZT^n(o) = 0$, if we can show that $(j/m)^n b_j (m,r/m)$ is finite when $m \to \infty$, then $g(1,r) = 0$ which gives the desired result for $r \neq 0$.

Indeed, for sufficiently large $m$, $(m+n)/m \leq 2$ and $a_i(r/m)$, $1 \leq i \leq n-1$ are smaller than a constant $N>0$. Hence, for $1\leq j \leq n-1$,
\begin{align*}
|(j/m)^n b_j (m,r/m)|
&\leq (n/m)^n \sum_{k=m-n+1}^{m-j} a_{m-k-j}(r/m) \sum_{k_1 =0}^k e^{-k_1 r/m} \binom{k-k_1 + n-2}{n-2} \\
&\leq (n/m)^n \sum_{k=m-n+1}^{m-j} a_{m-k-j}(r/m) \sum_{k_1 =0}^k e^{-k_1 r/m} (m+n)^{n-2} \\
&=(m+n)^{n-2} (n/m)^n \sum_{k=m-n+1}^{m-j} a_{m-k-j}(r/m)\frac{|1-e^{-(k+1)r/m|}}{|1-e^{-r/m}|} \\
&\leq 2^{n-2}n^n N (1/m)^2 \sum_{k=m-n+1}^{m-j}\frac{|1-e^{-(k+1)r/m}|}{|1-e^{-r/m}|} \\
&\leq 2^{n-2}n^n N (1/m)^2 \frac{(n-j)\max\{1,e^{-r}\}}{|1-e^{-r/m}|}.
\end{align*}
Note that $(1/m)^2 \frac{1}{|1-e^{-r/m}|} \to 0$ when $m \to \infty$. Hence $(j/m)^n b_j (m,r/m) \to 0$ when $m \to \infty$.
\end{proof}

For $0 < \lambda < 1$, let $H_\lambda$ be the hyperplane through the origin with normal vector $(1-\lambda) e_1- \lambda e_2$. Since $Z: \lpoly \to C(\R^n)$ is a simple valuation,
\begin{align}\label{val3}
ZT^{n} (x) = Z (T^{n} \cap H_\lambda ^-) (x) + Z (T^{n} \cap H_\lambda ^+) (x), ~~x \in \R^n.
\end{align}
We define $\phi_1,\phi_2 \in \GL$ by
$$\phi_1 e_1 = \lambda e_1 + (1-\lambda) e_2,~\phi_1 e_2 =  e_2,~\phi_1 e_i = e_i,~\text{for}~3 \leq i \leq n,$$
and
$$\phi_2 e_1 = e_1,~\phi_2 e_2 = \lambda e_1 + (1-\lambda) e_2,
\phi_2 e_i =  e_i,~\text{for}~3 \leq i \leq n.$$
Note that $T^{n}\cap H_\lambda ^- = \phi_1 T^{n}$, $T^{n}\cap H_\lambda ^+ = \phi_2 T^{n}$. The $\GL$ covariance of $Z$ and valuation relation (\ref{val3}) show that
\begin{align*}
ZT^{n} (x) = \lambda ZT^{n} (\phi_1 ^t x) + (1-\lambda) ZT^{n} (\phi_2 ^t x)
\end{align*}
Let $f(\cdot)=ZT^n(\cdot)$. We have
\begin{align}\label{30}
f(x) = \lambda f(\phi_1 ^t x) + (1-\lambda) f(\phi_2 ^t x)
\end{align}
for every $0< \lambda <1$ and $x = (x_1,\dots,x_n)^t \in \R^n$, where $\phi_1 ^t x = (\lambda x_1 + (1-\lambda)x_2, x_2, x_3,\dots, x_n)^t$ and $\phi_2 ^t x = (x_1, \lambda x_1 + (1-\lambda)x_2,x_3,\dots,x_n)^t$.

\begin{lem}\label{lem:sim}
Let $n \geq 2$ and let the function $f \in C(\R^n)$ satisfy the following properties. \\
(\rmnum{1}) $f$ satisfies the functional equation (\ref{30}); \\
(\rmnum{2}) For every even permutation $\pi$ and $x \in \R^n$,
$$f(x)=f(\pi x).$$
If $f(r e_1) =0$ for every $r \in \R$, then
$$f(x)=0$$
for every $x \in \R^n$.
\end{lem}
\begin{proof}
We prove the statement by induction on the number $m$ of coordinates of $x$ not equal to zero. By property (\rmnum{2}), we can assume that the first $m$ coordinates of $x$ are not equal to zero.

It is trivial that the statement is true for $m=1$. Assume that the statement holds true for $m-1$. We want to show that
\begin{align}\label{ind}
f(x_1 e_1 + \dots +x_m e_m)=0
\end{align}
for all the $x_1,\dots,x_m$ not zero.
For $x_1 > x_2 > 0$ or $0 > x_2 >x_1$, taking $x=x_1 e_1 + x_3 e_3 + \dots +x_m e_m$, $\lambda = \frac{x_2}{x_1}$ in (\ref{30}), we get
\begin{align}\label{38}
&f (x_1e_1 + x_3 e_3 + \dots +x_m e_m)  \nonumber \\
 &= \frac{x_2}{x_1}f (x_2 e_1 + x_3 e_3 + \dots +x_m e_m)  + \left(1-\frac{x_2}{x_1}\right) f(x_1 e_1 + x_2 e_2 +x_3 e_3 + \dots +x_m e_m).
\end{align}
For $x_2 > x_1 > 0$ or $0 > x_1 >x_2$,
taking $x=x_2 e_2 + x_3 e_3 + \dots +x_m e_m$, $1-\lambda = \frac{x_1}{x_2}$ in (\ref{30}), we get
\begin{align}\label{40}
&f (x_2e_2 + x_3 e_3 + \dots +x_m e_m) , \nonumber \\
&= \left(1- \frac{x_1}{x_2} \right)f (x_1 e_1 +x_2 e_2 + x_3 e_3 + \dots +x_m e_m) + \frac{x_1}{x_2}f(x_1 e_2 +x_3 e_3 + \dots +x_m e_m).
\end{align}
For $x_1 >0 > x_2$ or $x_2 > 0 > x_1$, taking $0 < \lambda = \frac{x_2}{x_2 - x_1} <1$ and $x= x_1 e_1 + x_2 e_2 + x_3 e_3 + \dots +x_m e_m$ in (\ref{30}), we get
\begin{align}\label{23}
&f(x_1 e_1 + x_2 e_2 + x_3 e_3 + \dots +x_m e_m) \nonumber \\
= & \frac{x_2}{x_2 - x_1} f(x_2 e_2 + x_3 e_3 + \dots +x_m e_m) + \frac{-x_1}{x_2 - x_1}f(x_1 e_1 + x_3 e_3 + \dots +x_m e_m).
\end{align}
Now, combined with the induction assumption and the continuity of $f$, (\ref{38}), (\ref{40}) and (\ref{23}) show that (\ref{ind}) holds true.
\end{proof}

\begin{proof}[Proof of Theorem \ref{thm:con}]
By Theorem \ref{thm:lapcon}, $c\mathcal{L}$ is a continuous, $\GL$ covariant and logarithmic translation covariant valuation on $\cb$.

Now we turn to the reverse statement. Since $Z: \cb \to C(\R^n)$ is a continuous, $\GL$ covariant and logarithmic translation covariant valuation, Lemma \ref{lem:ext} allows us to extend this valuation to a simple, $\GL$ covariant and logarithmic translation covariant valuation on  $\lcb$. Hence Lemma \ref{lem:cube} gives that there exists a constant $c \in \R$ such that
\begin{align*}
&ZC^n(re_1)=c \mathcal{L}C^n(re_1)
\end{align*}
for every $r \in \R$.
Now define $Z': \mathcal{P}^n \to C(\R^n)$ by
$$Z'P(x)=ZP(x)-c\mathcal{L}P(x), ~~x\in \R^n.$$
It is easy to see that $Z'$ is also a simple, $\GL$ covariant and logarithmic translation covariant valuation on  $\lcb$. Also $Z'C^n(re_1)=0$ for every $r \in \R$. Applying Lemma \ref{lem:inex} (for $Z'$) and Lemma \ref{lem:sim} (for $f=Z'T^n$) we get
$$Z'T^n=0.$$
Now using the inclusion-exclusion principle and the $\GL$ covariance and the simplicity of $Z'$ again, we have
$$Z'P=0$$
for every $P \in \mathcal{P}^n$ since every $P \in \mathcal{P}^n$ can be dissected into finite pieces of $\GL$ (with positive determinant) transforms and translations of $T^n$. Hence
$$ZP=c\mathcal{L}P$$
for every $P \in \mathcal{P}^n$. Since both $Z$ and $\mathcal{L}$ are continuous on $\cb$,
$$ZK=c\mathcal{L}K$$
for every $K \in \cb$.
\end{proof}


\subsection{Laplace transforms on functions}\label{sec:fun}
We first consider indicator functions of Borel sets.
\begin{lem}\label{thm:char}
    If a map $z:\la\rightarrow\conr$ is a continuous,
    $\GL$ covariant and logarithmic translation covariant valuation,
    then there exists a continuous function $h$ on $\R$ satisfying \eqref{eqn:zero} and \eqref{eqn:growth} such that
    \[z(\a\kik{E})(x)=h(\a)\int_{\R^n} \kik{E}(y) \me^{-\inp{x}{y}}dy\]
    for every $\a\in\R$, $x\in\R^n$ and bounded Borel set $E\subset\R^n$.
\end{lem}
\begin{proof}
    For any $\alpha \in \R$,
    define $Z_{\alpha }:\cb\rightarrow\conr$ by
    \[Z_\alpha K=z(\alpha \kik{K})\]
    for every $K\in\cb$.
    Lemma \ref{lem:rel} shows that $Z_\alpha$ is a continuous,
    $\GL$ and logarithmic translation covariant valuation on $\cb$.

    Therefore, by Theorem \ref{thm:con}, there exists a function $h \in\R$ such that
    \begin{align*}
    z(\alpha \kik{K})(x)=Z_\alpha K(x)&=h(\alpha) \lt K (x)\\
    &=h(\alpha) \int_K\me^{-\inp{x}{y}}dy
    = h(\a)\int_{\R^n} \kik{K}(y) \me^{-\inp{x}{y}}dy
    \end{align*}
    for every $K\in\cb$ and $x\in\R^n$. The function $h$ is continuous since $z$ is continuous and $\norm{\alpha_i \kik{K} - \alpha \kik{K}} \to 0$ whenever $\alpha_i \to \alpha$. $z(0) = 0$ (see \eqref{eq:z0}) gives that $h(0)=0$.
    In particular, this representation holds on $\upar$, the set of finite union of cubes, by the inclusion-exclusion principle.

Now we consider Borel sets in $\R^n$.
For each bounded Borel set $E\subset\R^n$, there exists a sequence $\set{K_i}$ in $\upar$
such that $\alpha \kik{K_i}\rightarrow \alpha \kik{E}$ in $L^1_c(\R^n)$ for every $\alpha \in \R$ as $i\rightarrow\infty$.
Moreover, for every $x \in \R^n$,
\begin{align*}
0 \leq \lim_{i\rightarrow\infty}\left|\int_{\R^n}\left(\kik{K_i} (y)-\kik{E} (y)\right) \me^{-\inp{x}{y}}dy\right|
&\leq \lim_{i\rightarrow\infty}\int_{\R^n} |\kik{K_i} (y)-\kik{E} (y)| \me^{-\inp{x}{y}} dy \\
&\leq \max_{y \in E}\me^{-\inp{x}{y}} \lim_{i\rightarrow\infty}\int_{\R^n} |\kik{K_i} (y)-\kik{E} (y)|dy \\
&=0.
\end{align*}
Thus, the continuity of $z$ gives that
\begin{align*}
z(\a\kik{E})(x) = \lim_{i\rightarrow\infty}z(\a\kik{K_i})(x)
&= h(\alpha) \lim_{i\rightarrow\infty}\int_{\R^n} \kik{K_i} (y) \me^{-\inp{x}{y}}dy \\
&= h(\alpha) \int_{\R^n} \kik{E} (y) \me^{-\inp{x}{y}}dy.
\end{align*}

The last step is to show that $h$ satisfies \eqref{eqn:growth}.
If $h$ does not satisfy \eqref{eqn:growth},
then there exists a sequence $\set{\a_j}$ in $\R \setminus \{0 \}$ (since $h$ satisfies \eqref{eqn:zero}) such that
\begin{align}\label{eqn:h}
|h(\a_j)|>2^i\abs{\a_j}.
\end{align}


Set $E_j=\left[0,2^{-j} \left/ \abs{\a_{j}} \right. \right] \times [0,1]^{n-1}$.
We have
\[\int_{E_j}dy=\frac{2^{-j}}{\abs{\a_{j}}}.\]
Let $g_j=\a_{j} \kik{E_j}$ and $f\equiv0$. Clearly $g_j, f \in \la$. Since
\begin{align*}
\int_{\R^n} \abs{g_j(y)} dy = \abs{\a_{j}} \int_{E_j}dy = 2^{-j} \to 0
\end{align*}
when $j \to \infty$. Hence $\|g_j - f\| \to 0$. The continuity of $z$ now implies that
\begin{align*}
z(f)(o) = \lim\limits_{j \to \infty} z(g_j) (o) dy.
\end{align*}
On the other hand, $z(f)(o) = 0$ (see \eqref{eq:z0}) and the above statement gives that
$$z(g_j) (o) = h(\a_{j}) \int_{\R^n} \kik{E_j}(y) dy$$
However,
since $h$ satisfies \eqref{eqn:h}, we obtain
\begin{align*}
0 = |z(f)(o)| = \lim_{j\to\infty} |z(g_j) (o)|
    &= \lim_{j\to\infty} \abs{h(\a_{j})} \int_{E_j}dy \\
    &= \lim_{j\to\infty} \frac{\abs{h(\a_{j})}}{2^{j}\xabs{\a_{j}}} \\
	&= \limsup_{j\to\infty} \frac{\abs{h(\a_{j})}}{2^{j}\xabs{\a_{j}}} \\
	&\geq 1.
\end{align*}
It is a contradiction.
Hence $h$ satisfies \eqref{eqn:growth}.
\end{proof}

Next, we deal with simple functions
\begin{lem}\label{thm:sf}
	Let $z:\la\to\conr$ be a valuation.
	Suppose that there exists a continuous function $h$ on $\R$ satisfying \eqref{eqn:zero} and \eqref{eqn:growth} such that
	\[z(\a\kik{E})(x)=h(\alpha) \int_{\R^n} \kik{E} (y) \me^{-\inp{x}{y}}dy\]
	for every $\a\in\R$, $x\in\R^n$ and bounded Borel set $E\subset\R^n$.
	Then
	\[z(g)= \int_{\R^n}(h\circ g)(y)\me^{-\inp{x}{y}}dy\]
	for every simple function $g\in\la$.
\end{lem}
\begin{proof}
Let $g\in\la$ be a simple function. We can write $g=\sum_{i=1}^m\a_i\kik{E_i}$, where $E_1,\ldots,E_m$ are disjoint bounded Borel sets and $\a_1,\ldots,\a_m\in\R$.
Hence
\begin{align}\label{eqn:sf}
z(g) = z\Big(\sum_{i=1}^m\a_i\kik{E_i}\Big) = z ((\a_1\kik{E_1}) \vee \dots \vee (\a_m\kik{E_m})) =\sum_{i=1}^mz(\a_i\kik{E_i}),
\end{align}
where the last equation is the inclusion-exclusion principle for valuations on the lattice $(\la,\vee,\wedge)$.
	
Since $h\circ g=\sum_{i=1}^mh(\a_i)\kik{E_i}$,
by \eqref{eqn:sf}, we obtain
\begin{align*}
z(g) = \sum_{i=1}^mz(\a_i\kik{E_i}) &= \sum_{i=1}^m h(\alpha_i) \int_{\R^n} \kik{E_i} (y) \me^{-\inp{x}{y}}dy\\
&= \int_{\R^n} \sum_{i=1}^m h(\alpha_i) \kik{E_i} (y) \me^{-\inp{x}{y}}dy = \int_{\R^n}  (h\circ g)(y) \me^{-\inp{x}{y}}dy.
\end{align*}
\end{proof}

Finally, we are ready to prove Theorem \ref{thm:laplace}.

\begin{proof}[Proof of Theorem \ref{thm:laplace}]
Theorem \ref{thm:prop} shows that $f \mapsto \mathcal{L}(h \circ f)$ is a continuous, $\GL$ covariant and logarithmic translation covariant valuation.
It remains to show the reverse statement.

For a nonnegative function $f\in\la$, there exists an increasing sequence of nonnegative simple functions $\set{g_k} \subset \la$ such that $g_k \uparrow f$ pointwise.
The monotone convergence theorem gives that $\norm{g_k-f} \to 0$. Note that every function $f \in \la$ can be written as $f = f_+ - f_-$, where
\begin{align*}
f_+ = \begin{cases}
f(x), &x \in \set{f \geq 0} \\
0, &x \in \set{f < 0}
\end{cases},~~~
f_- = \begin{cases}
0, &x \in \set{f \geq 0} \\
-f(x), &x \in \set{f < 0}
\end{cases}.
\end{align*}
Hence the above statement gives that there exists a sequence of simple functions $\set{g_k} \subset \la$ such that $g_k \to f$ pointwise and $\norm{g_k-f} \to 0$ by the triangle inequality. Moreover, the increasing sequence $|g_k(x)| \uparrow |f(x)|$ for every $x \in \R^n$.
    Due to the continuity of $z$, Lemma \ref{thm:char} and Lemma \ref{thm:sf}, we have
    \begin{equation}\label{eqn:yiban}
    	z(f)(x) = \lim_{k\to\infty}z(g_k)(x) = \lim_{i\to\infty}\int_{\R^n}(h\circ g_k)(y)\me^{-\inp{x}{y}}dy,
    \end{equation}
    where $h$ is a continuous function satisfying \eqref{eqn:zero} and \eqref{eqn:growth}.
    Therefore,
    \begin{align*}
    |h \circ g_k| \leq \gamma |g_k| \leq \gamma |f|.
    \end{align*}
    The dominated convergence theorem, the continuity of $h$, and \eqref{eqn:yiban} now yield
    \begin{align*}
    z(f) &= \lim_{i\to\infty}\int_{\R^n}(h\circ g_k)(y)\me^{-\inp{x}{y}}dy \\
    &= \int_{\R^n}(h\circ f)(y)\me^{-\inp{x}{y}}dy \\
    &= \lt(h\circ f).
    \end{align*}

\end{proof}


\section*{Acknowledgement}
\addcontentsline{toc}{section}{Acknowledgement}

The work of the first author was supported in part by the National Natural Science Foundation of China (11671249) and the Shanghai Leading Academic Discipline Project (S30104).
The work of the second author was supported in part by
the European Research Council (ERC) Project 306445 and
the Austrian Science Fund (FWF) Project P25515-N25.



\end{document}